\documentclass[11pt]{article}
\usepackage{url,color}
\usepackage{graphicx,cite}
\usepackage{bbm,epsfig,epstopdf}
\usepackage{enumerate,lineno}
\usepackage{mathrsfs,amsfonts,amsmath,amsthm,amssymb}
\usepackage[normalem]{ulem}
\usepackage{todonotes}
\usepackage{CJK}
\usepackage{array}
\usepackage{xcolor}
\theoremstyle{plain}
\newtheorem{theo}{Theorem}[section]
\newtheorem{cor}[theo]{Corollary}
\newtheorem{rem}[theo]{Remark}
\newtheorem{defi}[theo]{Definition}
\newtheorem{lemma}[theo]{Lemma}

\newtheorem{prop}[theo]{Proposition}
\newtheorem{ex}[theo]{Example}

\newtheorem{op}[theo]{Open Problem}

\usepackage{algpseudocode}
\usepackage[Algorithmus,section]{algorithm}
\algnewcommand{\IfOneRow}[1]{\State\algorithmicif\ #1,}
\algnewcommand{\EndifOneRow}{}

\algdef{SE}[SUBALG]{Indent}{EndIndent}{}{\algorithmicend\ }%
\algtext*{Indent}
\algtext*{EndIndent}

\usepackage{caption}
\makeatletter
\renewcommand{\ALG@name}{Algorithm}
\makeatother

\numberwithin{theo}{section}
\numberwithin{equation}{section}
\numberwithin{table}{section}
\numberwithin{figure}{section}

\allowdisplaybreaks

\DeclareFixedFont{\ttb}{T1}{txtt}{bx}{n}{12} 
\DeclareFixedFont{\ttm}{T1}{txtt}{m}{n}{12}  

\usepackage{color}
\definecolor{deepblue}{rgb}{0,0,0.5}
\definecolor{deepred}{rgb}{0.6,0,0}
\definecolor{deepgreen}{rgb}{0,0.5,0}

\usepackage{listings}

\newcommand\pythonstyle{\lstset{
language=Python,
basicstyle=\ttm,
morekeywords={self},              
keywordstyle=\ttb\color{deepblue},
emph={MyClass,__init__},          
emphstyle=\ttb\color{deepred},    
stringstyle=\color{deepgreen},
frame=tb,                         
showstringspaces=false
}}

\lstnewenvironment{python}[1][]
{
\pythonstyle
\lstset{#1}
}
{}

\newcommand\pythoninline[1]{{\pythonstyle\lstinline!#1!}}

\usepackage{float}
\usepackage{tocloft}
\usepackage{amsfonts,amsmath,amssymb}
\usepackage{multirow, verbatim}
\usepackage{shadow,enumerate}
\usepackage{makeidx} 
\usepackage{graphicx}
\usepackage{color}

\newcommand{\mb}[1]{#1}
\newcommand{\vF}[1]{\mathbb{F}_2^{#1}}

\usepackage{mathrsfs}

\usepackage{multirow,bm}
\usepackage{hhline}
\usepackage{array}
\usepackage{makecell}
\usepackage{multirow,bm}
\usepackage{hhline}
\usepackage{array}
\usepackage{makecell}

\def\cD{{\mathcal D}}
\def\cM{{\mathcal{M}}}
\newcommand{\B}{\mathcal{B}}
\newcommand{\F}{\mathbb{F}}

\newenvironment{polynomial}
{\par\vspace{\abovedisplayskip}%
	\setlength{\leftskip}{\parindent}%
	\setlength{\rightskip}{\leftskip}%
	\medmuskip=4mu plus 2mu minus 2mu
	\binoppenalty=0
	\noindent$\displaystyle}
{$\par\vspace{\belowdisplayskip}}

\usepackage{booktabs}
\usepackage[colorlinks=true,citecolor=blue,linkcolor=blue,urlcolor=blue,bookmarks,bookmarksopen,bookmarksdepth=3,backref=page]{hyperref}
\usepackage{bookmark}
\bookmarksetup{
	numbered, 
	open,
}

\renewcommand*{\backref}[1]{}
\renewcommand*{\backrefalt}[4]{%
	\ifcase #1 (Not cited.)%
	\or        (Cited on page~#2.)%
	\else      (Cited on pages~#2.)%
	\fi}
\title{Permutations satisfying $(P_1)$ and $(P_2)$  properties\\ and  $\ell$-optimal bent functions}
\author{\Large Sadmir Kudin$^1$, Enes Pasalic$^1$, Alexandr Polujan$^2$,
	 Fengrong Zhang$^{3}$\vspace{0.4cm} \\
	\small $^1$ University of Primorska, FAMNIT \& IAM, Glagolja\v{s}ka 8, 6000 Koper, Slovenia\ \\ \small \{\tt  sadmir.kudin@iam.upr.si, \tt enes.pasalic6@gmail.com\}\vspace{0.4cm}\\
	\small$^2$ Otto-von-Guericke-Universit\"{a}t, Universit\"{a}tsplatz 2, 39106, Magdeburg, Germany\ \\ \small \tt \{alexandr.polujan@gmail.com,alexandr.polujan@ovgu.de\} \vspace{0.4cm}\\
	\small	$^3$ School of Cyber Engineering, Xidian University, Xi'an 710071, China \ \\
	 \small \tt zhfl203@163.com
}
\date{}
\begin{document}
	
\maketitle
\begin{abstract}
	 An important classification of permutations over $\F_2^m$, suitable for constructing Maiorana-McFarland bent functions on $\F_2^m \times \F_2^m$ with the unique $\mathcal{M}$-subspace of maximal dimension, was recently considered in~\cite{PPKZ2023}. More precisely, two properties called $(P_1)$ and $(P_2)$ were introduced and a generic method of constructing permutations having the property $(P_1)$ was presented, whereas no such results were provided related to the $(P_2)$ property. In this  article, we provide a deeper insight on these properties, their mutual relationship, and specify some explicit classes of permutations having these properties. Such permutations are then employed to generate a large variety of bent functions outside the completed Maiorana-McFarland class $\cM^\#$.
	 We also introduce $\ell$-optimal bent functions as bent functions with the lowest possible linearity index; such functions can be considered as opposite to Maiorana-McFarland bent functions. We give explicit constructions of $\ell$-optimal bent functions within the $\cD_0$ class, which in turn can be employed in certain secondary constructions of bent functions~\cite{ZPBW23} for providing even more classes of bent functions that are provably outside $\cM^\#$. Moreover, we demonstrate that a certain subclass of $\cD_0$ has an additional property of having only 5-valued spectra decompositions, similarly to the only result in this direction concerning monomial bent functions~\cite{Decom}. Finally, we generalize the so-called ``swapping variables'' method introduced in \cite{PPKZ2023} which then allows us to specify much larger families of bent functions outside $\cM^\#$ compared to \cite{PPKZ2023}. In this way, we give a better explanation of the origin of bent functions in dimension eight, since the vast majority of them is outside $\mathcal{M}^\#$,  as indicated in~\cite{LangevinL11}. 
	\end{abstract}
	
\section{Introduction}
Bent functions, introduced by Rothaus in the mid-1960s~\cite{Rot}, are well-known combinatorial objects that play an important role in the construction of various discrete structures, including difference sets, combinatorial designs, and strongly regular graphs~\cite{Carlet2021,Dillon,Mesnager}. Thanks to their exceptional differential properties and perfect nonlinearity, bent functions found many applications in cryptography~\cite{Carlet2021,Mesnager}. For example, the cryptographic hash function HAVAL utilizes Boolean functions derived from bent functions in six variables~\cite{ZPS1993}. Additionally, some components of the block ciphers CAST-128 and CAST-256 were designed using the CAST design procedure~\cite{Adams1997}, which also incorporates bent functions. Moreover, they play an important role in the  design of BISON  (for Bent whItened Swap Or Not) -- the first practical instance of the Whitened Swap-Or-Not construction~\cite{Bison}.

Probably the most important class of bent functions is the Maiorana-McFarland class~\cite{MM73} $\mathcal{M}$, i.e., the set of bent functions on $\F_2^m\times\F_2^m$ of the form  $f(x,y)=x \cdot \pi(y) + h(y)$ , where $\pi$ is a permutation of $\F_2^m$ and $h$ is an arbitrary Boolean function on $\F_2^m$. Due to the flexibility of the choice of a permutation $\pi$ and a Boolean function $h$ on $\F_2^n$, one can design bent functions with  desired algebraic properties~\cite{CarletMesnager2016}. The primary cryptographic disadvantage of this construction is that any Maiorana-McFarland bent function can be expressed as a concatenation of $2^m$ affine functions over $\F_2^m$, which can be exploited in attacks~\cite[p. 295]{Carlet2021}. Since this property is invariant under equivalence, there is an essential necessity to construct bent functions outside the completed Maiorana-McFarland class~$\mathcal{M}^\#$, i.e., the set of all bent functions that are EA-equivalent to those in the $\mathcal{M}$ class. For the recent works on this subject, we refer to~\cite{Cclass_DCC,OutsideMM,BapicPasalic,Kudin2021,Kudin2022,ISIT,ISIT_FULL,Bent_Decomp2022,PPKZ2023,PPKZ_CCDS}.

 Dillon, in his thesis~\cite{Dillon}, proved that a given bent function $f\colon\F_2^n \to \F_2$ (where $n$ is necessarily even) belongs to the $\cM^\#$ class if and only if $D_aD_b f(x)=0$ for all $a,b\in V$ (and for all $x \in \F_2^n$), for at least one $n/2$-dimensional vector subspace $V$ of $\F_2^n$ (see also Lemma~\ref{lem M-M second} for details). Vector subspaces $V$ of $\F_2^n$ such that for any two elements $a,b \in V$ the second-order derivative $D_aD_bf$ is the zero function on $\F_2^{n}$, were called \textit{$\mathcal{M}$-subspaces} in~\cite{Polujan2020}.  The algebraic properties of $\mathcal{M}$-subspaces attracted more attention only recently in a few works, e.g., in~\cite{Polujan2020,PolujanPhD,PPKZ2023}. In this article, we further develop the theory of $\mathcal{M}$-subspaces, as a highlight, we provide generic constructions of bent functions having only trivial $\mathcal{M}$-subspaces, i.e., those of dimension at most one. Such functions can be seen as the opposite of Maiorana-McFarland functions, since they can not be represented as a concatenation of affine functions defined on a ``large'' vector space. The constructions of such functions are very limited, and the known examples stem from monomial bent functions~\cite{Decom} thanks to their strong  multiplicative properties. We, on the other hand, provide many such examples employing additive properties of the involved building blocks. 
 
The first main aim of this article is to provide further analysis of special classes of permutations on $\F_2^m$ that give rise to bent functions on $\F_2^m \times \F_2^m$ in the $\cM$ class, which have a unique $m$-dimensional $\mathcal{M}$-subspace $\F_2^{m} \times \{0_{m}\}$. Such functions were recently~\cite{PPKZ2023} shown to be important primitives in the design of bent functions outside $\mathcal{M}^\#$ using the concatenation~\cite{Decom} $f=f_1||f_2||f_3||f_4$ of four functions $f_1,f_2,f_3,f_4$ on $\F_2^n$. Recently, it was shown that the Maiorana-McFarland bent functions $f(x,y)=x \cdot \pi(y) + h(y)$ on $\F_2^m\times\F_2^m$, with the unique $\mathcal{M}$-subspace of dimension $m$, can be constructed from a permutation $\pi$ of $\F_2^m$ satisfying the so-called \eqref{eq: P1} and \eqref{eq: P2} properties. Whereas the property \eqref{eq: P1} means that $D_aD_b \pi$ is not the constant zero function, for any linearly independent  $a, b \in \F_2^m$, the property  \eqref{eq: P2} appears to be less strict since the maximal dimension of a subspace $S$ for which  $D_aD_b \pi=0_m$, for all $a,b\in S$, is at most $m-1$; however in this  latter  case an additional condition must be satisfied. It was already noticed in \cite{PPKZ2023}  that 34 equivalence classes of quadratic permutations on $\F_2^5$ (out of 75 in total) satisfy the property~\eqref{eq: P2}, while only two of them have \eqref{eq: P1}. In this context, we provide a generic method of preserving the property \eqref{eq: P2} on larger variable spaces which significantly increases the cardinality of bent functions in $\cM$ that admit a unique $\mathcal{M}$-subspace of maximal dimension. Moreover, we formally show that the property \eqref{eq: P1} implies \eqref{eq: P2} and simplify the sufficient conditions related to the latter property.
  
The second main aim of this article is to provide  constructions of bent functions  $f\in\mathcal{B}_n$ with maximal dimension of $\mathcal{M}$-subspaces being equal to one. Such functions can be considered as opposite to the Maiorana-McFarland functions on $\F_2^n$, which always posses at least one $\mathcal{M}$-subspace of the maximal possible dimension $n/2$. Due to the recent results based on the analysis of $\mathcal{M}$-subspaces \cite{PPKZ2023,BFAExtended,ISIT,ISIT_FULL}, the design methods of bent functions outside $\cM^\#$ (equivalently having the linearity index less than $m$ for a bent function on $\F_2^{2m}$) are quite efficient without any complicated conditions to be satisfied. However, a little is known about the constructions of bent functions outside $\cM^\#$ with a prescribed maximal dimension of $\mathcal{M}$-subspaces. In the extreme case, the linearity index of a bent function $f$ equals to one (implying that $D_aD_b f =0$ only if $\dim(\langle a, b \rangle)=1$) which was initially considered in~\cite[Lemma 3]{Decom}. We call such bent functions $\ell$-optimal and show that such bent functions can be specified within the $\cD_0$ class, whose members are of the form $f(x,y)=x \cdot \pi(y) + \delta_0(x)$ (where $\delta_0(x)=\prod_{i=1}^m(x_i+1)$) for certain permutations $\pi$ over $\F_2^m$. More precisely,  to obtain $\ell$-optimality the permutation $\pi$  must satisfy  the \eqref{eq: P1} property and moreover its components do not admit linear structures. Such permutations can be identified among certain non-quadratic monomial mappings (notice that the APN permutations always satisfy \eqref{eq: P1}). However, specifying  other construction methods of such permutations is  left as an open problem. We also note that $\ell$-optimal bent functions might also have an additional property of having only 5-valued spectra decompositions and one such class is identified, see Theorem \ref{theo:5valueddecomp}. Actually, we demonstrate  that the linearity index of $f$ and  the dual bent function $f^*$ are not necessarily the same, which also implies that $\ell$-optimality of $f$ does not necessarily induce the property of having  5-valued spectra decompositions only. 
  
  Additionally, we consider the so-called ``swapping variables'' approach considered in \cite{PPKZ2023} for the purpose of specifying efficient designs of bent functions outside $\cM^\#$ using bent functions $f_i$  in $\cM$, when the concatenation of the form $f=f_1||f_2||f_3||f_4$ is considered. We provide a generalization of this method, see Theorem \ref{theo bent and q bent} and Corollary \ref{cor:swapping}, which significantly extends the cardinality of families of bent functions outside $\cM^\#$.  This approach, using bent functions $f_i$  in $\cM$, is currently the most efficient design for specifying bent functions outside $\cM^\#$ on $\F_2^8$ which then contributes to our better understanding of the origin of bent functions for this particular dimension of the ambient space. 
  
  The rest of this article is organized as follows. In Section \ref{sec:pre}, we provide some relevant notations and definitions related to Boolean and bent functions in particular. In Subsection~\ref{sub: 2.1 MM}, we summarize some results on the Maiorana-McFarland bent functions and $\mathcal{M}$-subspaces, while in Subsection~\ref{sub: bent-4 cat and its prop} we consider decompositions and concatenations of bent functions. In Section \ref{sec: 3 refining}, we consider in detail the relationship between the properties \eqref{eq: P1} and \eqref{eq: P2} and we address further refinement of the latter property. In Section~\ref{sec: 4 P2 property}, we give a construction of permutations with the \eqref{eq: P2} property, thus providing a solution to~\cite[Open Problem 1]{PPKZ2023}. In Section~\ref{sec: 5 swapping}, a generalization of the ``swapping variables'' method is proposed along with the related design of bent functions outside $\cM^\#$. In Section \ref{sec: 6 loptimal}, we introduce the notion of $\ell$-optimal bent functions. In Subsection~\ref{sub: 6.1. l-optimality and D0}, we identify a certain class of $\ell$-optimal bent functions and in Subsection~\ref{sub: 6.2. 4-deom and D0} we consider in more detail those that have only 5-valued spectra decompositions. The paper is concluded in Section \ref{sec: 7 concl}.

\section{Preliminaries}\label{sec:pre}
Let $\mathbb{F}_2^n$ be the vector space of dimension $n$ over $\F_2$.  For $x=(x_1,\ldots,x_n)$ and $y=(y_1,\ldots,y_n)$ in $\mathbb{F}^n_2$, the scalar product over $\mathbb{F}_2$ is defined as $x\cdot y=x_1 y_1 + \cdots +  x_n y_n$. The Hamming weight of $x=(x_1,\ldots,x_n)\in \mathbb{F}^n_2$ is defined by $\operatorname{wt}(x)=\sum^n_{i=1} x_i$ and the all-zero vector with $n$ coordinates is denoted by $0_n=(0,0,\ldots,0)\in \mathbb{F}^n_2$. In certain cases, we equip $\F_2^n$ with the structure of the finite field $\left(\F_{2^{n}}, +, \cdot \right)$. In this case, the absolute trace of an element $x\in\F_{2^n}$ is given by $Tr(x) =\sum_{i=0}^{n-1} x^{2^{i}}$.

A Boolean function $f$ in $n$ variables is a mapping $f\colon\F_2^n\to\F_2$. The set of all Boolean functions in $n$ variables is denoted by $\mathcal{B}_n$.  Any Boolean function $f\in\mathcal{B}_n$ can be uniquely represented by the algebraic normal form (ANF) given by $f(x_1,\ldots,x_n)=\sum_{u\in \mathbb{F}^n_2}{\lambda_u}{(\prod_{i=1}^n{x_i}^{u_i})}$, where $x_i, \lambda_u \in \mathbb{F}_2$ and $u=(u_1, \ldots,u_n)\in \mathbb{F}^n_2$.
The algebraic degree $\deg(f)$ of $f\in\mathcal{B}_n$ is defined as the maximum Hamming weight of $u \in \F_2^n$, for which $\lambda_u \neq 0$ in its ANF. The first-order derivative of a function $f\in\mathcal{B}_n$, in the direction $a \in \F_2^n$, is the mapping $D_{a}f(x)=f(x+a) +  f(x)$. For $a,b\in\F_2^n$, the second-order derivative of a function $f\in\mathcal{B}_n$ is the mapping $D_{a}D_bf(x)=f(x+a+b) + f(x+a) + f(x+b) + f(x)$. 
An element $a\in\F_2^n$ is a linear structure of $f\in\mathcal{B}_n$ if $f(x+a)+f(x)=const.$, for all $x\in\F_2^n$. A function $f\in\mathcal{B}_n$ is said to have no linear structures if $0_n$ is the only linear structure of $f$.

The Walsh-Hadamard transform (WHT) of $f\in\mathcal{B}_n$ at $a\in\mathbb{F}^n_2$  is defined by $W_{f}(a)=\sum_{x\in \mathbb{F}_2^n}(-1)^{f(x) +  a\cdot x}$. Its inverse WHT, is in turn given by $	(-1)^{f(x)}=2^{-n}\sum_{a\in \mathbb{F}_2^n}W_f(a)(-1)^{a\cdot x}$. For even $n$, a function $f\in\mathcal{B}_n$ is called \emph{bent} if $W_f(u)=\pm2^{\frac{n}{2}}$ for all $u\in\F_2^n$. For a bent function $f\in\mathcal{B}_n$, the  Boolean function $f^*\in \mathcal{B}_n$ defined by $W_f(u)=2^{\frac{n}{2}}(-1)^{f^*(u)}$, for all $u\in\F_2^n$, is a bent function, called the \emph{dual} of $f$. 

For $m\ge2$, the mappings $F\colon\F_2^n\to\F_2^m$ are called vectorial functions. Every such function $F\colon\F_2^n\to\F_2^m$ can be uniquely represented as $F(x)=(f_1(x),\ldots,f_m(x))$, where Boolean functions $f_i  \in\mathcal{B}_n$ are called coordinate functions. The ANF of $F\colon\F_2^n\to\F_2^m$ is defined coordinate-wise, and $\deg(F):=\max\{\deg(f_i)\colon F=(f_1,\ldots,f_m)\}$. For $b\in\F_2^m \setminus \{0_m\}$, the component function $F_b\in\mathcal{B}_n$ of $F\colon\F_2^n\to\F_2^m$ is defined by $F_b(x)=b\cdot F(x)$, for all $x\in\F_2^n$. For vectorial functions, the definitions related to differential properties (e.g., derivatives, linear structures) can be essentially introduced by replacing $f\in\mathcal{B}_n$ in the corresponding definitions by $F\colon\F_2^n\to\F_2^m$.   
A function $F\colon\F_2^m\to \F_2^m$ is called \textit{almost perfect nonlinear (APN)} if, for all $a\in\F_2^m\setminus\{0_m\}$, $b\in\F_2^m$, the equation $F(x+a)+F(x)=b$ has 0 or 2 solutions $x\in\F_2^m$. Finally, we note that for the case $n=m$, we will frequently use the univariate representation over finite fields so that $F\colon\F_{2^n} \to \F_{2^n}$ is represented as $F(x)= \sum_i a_i x^i$, where $a_i \in \F_{2^n}$.

Boolean functions $f,f'\in\mathcal{B}_n$ are extended affine (EA) \textit{equivalent} if there exists an affine permutation $A$ of $\F_2^n$ and an affine function $l\in\mathcal{B}_n$ (i.e., $\deg(l)\le1$) such that $f\circ A + l= f'$. It is well-known that the bent property is preserved under extended-affine equivalence. This fact  essentially leads to the following definition. A class of bent functions $\mathit{B}_n \subset \mathcal{B}_n$ is \emph{complete} if it is globally invariant under EA-equivalence. 

Now, we introduce basic definitions and some fundamental results related to the completed Maiorana-McFarland class ($\mathcal{M}^\#$) of bent functions and bent 4-concatenation, which will be required later in the sections related to construction methods of bent functions outside $\mathcal{M}^\#$ using bent 4-concatenation.

\subsection{Maiorana-McFarland bent functions and  $\mathcal{M}$-subspaces}\label{sub: 2.1 MM}

The \textit{Maiorana-McFarland class} $\cM$ is the set of $n$-variable ($n=2m$) Boolean bent
functions of the form
\[
f(x,y)=x \cdot \pi(y)+ h(y), \mbox{ for all } x, y\in\F_2^m,
\]
where $\pi$ is a permutation on $\F_2^m$  and $h$ is an arbitrary Boolean function on $\F_2^m$. The smallest class that contains $\mathcal{M}$, that is globally EA-invariant,  is denoted  by $\mathcal{M}^\#$ and  is called the \textit{completed Maiorana-McFarland class}.
 Using the following criterion, one can analyze whether a given Boolean bent function $f\in\mathcal{B}_n$ belongs to $\mathcal{M}^\#$.
\begin{lemma} \cite[p. 102]{Dillon}\label{lem M-M second}
	Let $n=2m$. A Boolean bent function $f\in\mathcal{B}_n$ belongs to $\cM^{\#}$ if and only if
	there exists an $m$-dimensional linear subspace $V$ of $\F_2^n$ such that, for any $ a,  b \in V$, 
	$$ D_{a}D_{b}f(x)=f(x)+f(x+a)+f(x+b)+f(x+a+b)=0, \textit{ for all }x \in \F_2^n.$$
\end{lemma} 
Following the terminology in~\cite{Polujan2020}, we introduce the $\mathcal{M}$-subspaces of Boolean (not necessarily bent) functions in the following way.
\begin{defi}
	Let $f\in\mathcal{B}_n$ be a Boolean function. We call a vector subspace $V$ of $\F_2^n$  an $\mathcal{M}$-subspace of $f$, if we have that 	$ D_{a}D_{b}f=0$, for any $ a,  b \in V$. We denote by $\mathcal{MS}_r(f)$ the collection of all $r$-dimensional $\mathcal{M}$-subspaces of the function $f$ and by $\mathcal{MS}(f)$ the collection 
	$\mathcal{MS}(f):=\bigcup\limits_{r=1}^{n} \mathcal{MS}_r(f)$. The \emph{linearity index} $\operatorname{ind}(f)$ of a Boolean function $f\in\mathcal{B}_n$ is the maximal possible dimension of an $\mathcal{M}$-subspace of $f$, i.e., $\operatorname{ind}(f)=\max\limits_{U\in \mathcal{MS}(f)}\dim(U)$.
\end{defi} 
\begin{rem}
	For shortness of notation, we often drop the involved variable in the expression $D_aD_b f=0$, where $f\in\mathcal{B}_n$, $a,b\in\F_2^n$. In such cases, we actually mean that $D_aD_b f(x)=0$, for all $x \in \F_2^n$.
\end{rem}
The linearity index of a Boolean function $f\in\mathcal{B}_n$ is an invariant under EA-equivalence, see~\cite{Carlet2021,PolujanPhD}. Particularly, for a bent function $f\in\mathcal{B}_n$ it holds that $1\le \operatorname{ind}(f)\le n/2$. Bent functions achieving the upper bound with equality are exactly the bent functions in $\cM^\#$ by Lemma~\ref{lem M-M second}. 

In~\cite[Proposition 4.4]{Polujan2020}, it was shown that for a Boolean function $f\in\mathcal{B}_n$ the total number of $\mathcal{M}$-subspaces of a fixed dimension $r$ (that is $|\mathcal{MS}_r(f)|$) is invariant under EA-equivalence.  For every Maiorana-McFarland bent function $f(x,y)=x\cdot\pi(y)+h(y)$ on $\F_2^m\times\F_2^m$, the  subspace $\F_2^{m} \times \{0_m \}$ is an $\mathcal{M}$-subspace of maximal dimension, as observed by Dillon~\cite{Dillon}; it is called the \textit{canonical $\mathcal{M}$-subspace}. Note that in general, this vector space is not necessarily unique. For instance, for a bent function $f\in\mathcal{M}^\#$ on $\F_2^n$ it holds that $1\le|\mathcal{MS}_{n/2}(f)|\le\prod_{i=1}^{n/2} \left(2^i+1\right)$, see~\cite{Polujan2020}. The upper bound is achieved with equality if and only if $f$ is quadratic, see~\cite{PolujanPhD} and~\cite[Theorem 2]{Kolomeec17}. 

\subsection{Decomposing and concatenating bent functions}\label{sub: bent-4 cat and its prop}

Canteaut and Charpin~\cite{Decom} considered the \textit{4-decomposition} $f = (f_1, f_2, f_3, f_4)_V$ of a given bent function $f\in\mathcal{B}_{n+2}$ into four Boolean functions $f_1, \ldots, f_4 \in \mathcal{B}_{n}$ that are defined on the cosets of an $n$-dimensional subspace $V = \langle a, b\rangle^{\perp}$ of  $\F_2^{n+2}$, where for any linear subspace  $S\subset \mathbb{F}^n_2$, its orthogonal complement is defined as 
	$S^\perp=\{x\in \mathbb{F}^n_2: x\cdot y = 0,\textnormal{ for all } y\in S\}$.  Remarkably, for any bent function $f\in\mathcal{B}_{n+2}$ and any $n$-dimensional subspace $V = \langle a, b\rangle^{\perp}$ (where $a,b\in\F_2^{n+2}$ are linearly independent), all functions $f_i\in\mathcal{B}_n$ in the 4-decomposition $f = (f_1, f_2, f_3, f_4)_V$ are simultaneously bent, disjoint spectra semi-bent, or suitable 5-valued spectra functions. More precisely, for any $a \in \F_2^n$, we have $W_{f_i}(a) =\pm 2^{n/2}$, $W_{f_i}(a) \in \{0, \pm 2^{n/2+1}\}$, or $W_{f_i}(a) \in \{0,\pm 2^{n/2}, \pm 2^{n/2+1}\}$, respectively. For $V=\langle (0,\ldots,1,0),(0,\ldots,0,1)\rangle^\perp\subset\F_2^{n+2}$, a given function $f\in\mathcal{B}_{n+2}$ can be reconstructed from four functions $f_1,f_2,f_3,f_4\in\mathcal{B}_n$ using the \textit{4-concatenation} $f=f_1||f_2||f_3||f_4$, whose ANF is given by
\begin{equation}\label{eq:ANF_4conc}
f(x,y_1,y_2)=f_1(x) +  y_1(f_1 +  f_3)(x) +  y_2(f_1 +  f_2)(x) +  y_1y_2(f_1 +  f_2 +  f_3 +  f_4)(x).
\end{equation}
In this way, $f_1(x)=f(x,0,0), f_2(x)=f(x,0,1),f_3(x)=f(x,1,0)$ and $f_4(x)=f(x,1,1)$. When all $f_i \in \B_n$ are bent, we have that $f=f_1||f_2||f_3||f_4 \in \mathcal{B}_{n+2}$ is bent if and only if the \textit{dual bent condition} is satisfied~\cite{SHCF}, i.e., $f^*_1  +  f^*_2  +  f^*_3  +  f^*_4=1$. In this case, we call $f=f_1||f_2||f_3||f_4 \in \mathcal{B}_{n+2}$ a \emph{bent 4-concatenation}. For the recent construction methods of such functions, we refer to~\cite{PPKZ_CCDS}. Finally, we give  the expression of the second-order derivative $D_{a,b}f$ for $f=f_1||f_2||f_3||f_4\in\mathcal{B}_{n+2}$, where  $a=(a',a_1,a_2)$ and $b=(b',b_1,b_2)$ and $a',b' \in \F_2^n$ and $a_i,b_i \in \F_2$ (see~\cite[Eq. (I.2)]{PPKZ2023} for more detail):
\begin{equation}\label{eq:2ndderiv_conc correct}
	\begin{split}
		D_aD_bf(x,y_1,y_2)&=D_{a'}D_{b'}f_1(x)+ y_1D_{a'}D_{b'}f_{13}(x)+ y_2 D_{a'}D_{b'}f_{12}(x)+  y_1y_2D_{a'}D_{b'}f_{1234}(x) \\
		& + a_1D_{b'}f_{13}(x +  a')+  b_1 D_{a'}f_{13}(x +  b') +  a_2D_{b'}f_{12}(x +  a') +  b_2D_{a'}f_{12}(x +  b')  \\
		& +  (a_1y_2 +  a_2y_1 +  a_1a_2)D_{b'}f_{1234}(x +  a') +  (b_1y_2 +  b_2y_1 +  b_1b_2)D_{a'}f_{1234}(x +  b') \\
		 &  +  (a_1b_2 +  b_1a_2)f_{1234}  (x +  a' +  b'). 
	\end{split}
\end{equation}
Here,  the Boolean function $f_{i_1\ldots i_k}\in\mathcal{B}_n$ is defined by $f_{i_1\ldots i_k}:=f_{i_1} +  \cdots  +  f_{i_k}$. This expression together with Lemma~\ref{lem M-M second} will be later used to specify suitable $f_i\in\mathcal{B}_n$, so that $f\in\mathcal{B}_{n+2}$ is a bent function outside $\cM^\#$.

\section{Refining the \eqref{eq: P1} and   \eqref{eq: P2} properties  for permutations over $\F_2^m$}\label{sec: 3 refining}
In~\cite{PPKZ2023}, the authors specified algebraic properties of permutations $\pi$ of $\F_2^m$ which guarantee that a Maiorana-McFarland bent function $f(x,y)=x\cdot \pi(y)  + h(y)\in\mathcal{B}_{2m}$ admits exactly one $m$-dimensional $\mathcal{M}$-subspace. This feature is advantageous from the perspective of constructing bent functions $f=f_1||f_2||f_3||f_4\in\mathcal{B}_{2m+2}$ outside $\cM^\#$ from bent functions $f_i\in\mathcal{B}_{2m}$ inside $\cM^\#$, since in this case it is easier  to ensure that the second-order derivatives of $f$ do not vanish on any $(m+1)$-dimensional subspace of $\F_2^{2m+2}$. It was shown in \cite{PPKZ2023}, that the property of having a unique $\mathcal{M}$-subspace of maximal dimension for $f \in \mathcal{M}$ is directly related to the so-called properties~\eqref{eq: P1} and~\eqref{eq: P2} of a permutation $\pi$, which are defined below. 
\begin{theo}\label{theo: unique P1} \cite{PPKZ2023} Let $\pi$ be a permutation of $\F_2^m$ which has the following property:
	\begin{equation}\label{eq: P1} \tag{$P_1$}
		D_vD_w\pi\neq0_m \mbox{ for all linearly independent } v,w\in\F_2^m.
	\end{equation}
	Define $f\colon\F_2^{m}\times\F_2^{m} \to \F_2$ by $f(x,y)=x \cdot \pi(y) + h(y)$, for all $x,y \in \F_2^{m}$, where $h\colon\F_2^m \to \F_2$ is an arbitrary Boolean function. Then, the following hold:
	\begin{itemize}
		\item[1)] The permutation $\pi$ has no linear structures.
		\item[2)] The vector space $V=\F_2^m \times \{0_m \}$ is the only $m$-dimensional $\mathcal{M}$-subspace of $f$.
	\end{itemize}
\end{theo}

\begin{defi}\label{defi:p2}\cite{PPKZ2023}
	Let $\pi$ be a permutation of $\F_2^m$. Let $S$ be a subspace of $\F_2^m$ of dimension $m-k$, with $1 \leq k  \leq m-1$, such that $D_aD_b\pi=0_m$ for all $a,b \in S$. Then, $\pi$ satisfies the property ($P_2$) with respect to $S$ if there does not exist a vector subspace $V$ of $\F_2^m$  with $\dim(V) = k$ such that 
	\begin{equation}\label{eq: P2} \tag{$P_2$}
		v \cdot D_a\pi(y) = 0;  \textnormal{ for all } a \in S, \; \textnormal{ all } y \in \F_2^m  \; \textnormal{, and  for all } v \in V.
			\end{equation}
	If $\pi$ satisfies this property  with respect to any linear subspace $S$ of $\F_2^m$ of arbitrary   dimension $1 \leq \dim(S)  \leq m-1$, then we simply say that $\pi$  satisfies \eqref{eq: P2}.  
\end{defi}
\begin{rem}\label{rem:P2andlinstruc}
	Notice that the \eqref{eq: P2} property implies that $\pi$ has no linear structures. Indeed, assume that $\pi$ has a nonzero linear structure $a\in\F_2^m$, i.e., for some $z\in\F_2^m$ it holds that $D_{a}\pi(y)=z$, for all $y\in\F_2^m$. Set $V= \langle z \rangle^{\perp}.$ Then, $\dim(V)=m-1$ and 
	$$
	v \cdot D_a\pi(y) = v \cdot z =  0;  \textnormal{ for all } y \in \F_2^m  \; \textnormal{, and  for all } v \in V,
	$$
	hence $\pi$ does not satisfy the property~\eqref{eq: P2} with respect to the subspace $S=\langle a \rangle$ (the condition $D_aD_b\pi=0_m$, for all $a,b \in S$, is trivially satisfied).
\end{rem}
 
The following result gives the equivalence between the \eqref{eq: P2} property and the uniqueness of $m$-dimensional $\mathcal{M}$-subspace of $f(x,y)=x \cdot \pi(y)$. 
\begin{prop}\label{prop:suffcondunique}	\cite{PPKZ2023}
	Let $\pi$  be a non-affine permutation of $\F_2^m$ and $f(x,y)= x \cdot \pi(y)$ be a bent function on $\F_2^m \times \F_2^m$ in $\cM$.
	Then, the permutation $\pi$ has the property \eqref{eq: P2} if and only if the only $m$-dimensional $\mathcal{M}$-subspace of $f$ is $\F_2^m \times \{0_m\}$.
\end{prop}

The following remark regarding the property \eqref{eq: P1} will be used to analyze 4-decomposition of a certain subclass of functions in the $\cD_0$ class in Section~\ref{sec: 6 loptimal}.

\begin{rem} \label{rem:inverseP1}
	When $\pi$ is an APN permutation of $\F_2^m$ then it satisfies  the property \eqref{eq: P1} and consequently also \eqref{eq: P2}. Then, its inverse $\pi^{-1}$, which is also APN, satisfies both properties as well.
\end{rem}  

In the following statement, we provide an alternative characterization of the ~\eqref{eq: P2} property, which simplifies a specification of such permutations.
\begin{prop}\label{prop:Vs(f)}
Let $\pi$ be a permutation of $\F_2^m$ and let $S$ be a $k$-dimensional subspace of $\F_2^m$, $k\in \lbrace 1, 2, \ldots , m-1 \rbrace$, such that $D_aD_b \pi = 0_m$, for all $a,b \in S$. Denote by $V_S(f)$ the subspace of $\F_2^m$ generated by the set 
$ \lbrace D_a \pi(y) : a \in S \text{ and } y \in \F_2^m \rbrace.$ 
 Then, the permutation $\pi$ satisfies the property~\eqref{eq: P2} with respect to the subspace $S$ if and only if $\dim(V_S(f)) > \dim(S)=k$. 
\end{prop} 
\begin{proof} This follows from the fact that $\dim(V_S(f)) \leq k$ if and only if $\dim(V_S(f)^{\perp}) \geq m-k$. 
	\end{proof}
Nevertheless, using the following lemma, it is possible to further refine this property. 
\begin{lemma} \cite{KudinPhD,Kudin2022} \label{lem:algdegvec}
	Let $G: \F_2^m \to \F_2^t$ be a vectorial Boolean function. If there exists an $(m-k)$-dimensional subspace $H$ of $\F_2^m$ such that $D_aD_b G=0_t$ for all $a,b \in H$, then the algebraic degree of $G$ is at most $k+1$.
\end{lemma}
\begin{theo}\label{th:refineddegreeP2} For a permutation $\pi$ over $\F_2^m$, to satisfy the property  ~\eqref{eq: P2}, it is enough to verify that it satisfies \eqref{eq: P2} for all subspaces $S$ such that $\dim(S) \leq m-\deg(\pi)+1$, where $D_aD_b \pi =0_m$, for all $a,b \in S$. In particular, when $\deg(\pi)=m-1$, to verify the property~\eqref{eq: P2} it is enough to check that it satisfies \eqref{eq: P2} for all 2-dimensional subspaces and that it has no linear structures.
	
\end{theo}
\begin{proof}
	For a permutation $\pi$ of $\F_2^m$, Lemma~\ref{lem:algdegvec} implies that if there exists an $(m-k)$-dimensional subspace $S$ of $\F_2^m$, $k\in \lbrace 1, 2, \ldots , m-1 \rbrace$, such that $D_aD_b \pi = 0_m$, for all $a,b \in S$, then $\deg(\pi) \leq k+1$. Hence,
	 to check if a permutation satisfies the property~\eqref{eq: P2}, it is enough to check that it satisfies \eqref{eq: P2} for all subspaces $S$ such that $\dim(S) \leq m-\deg(\pi)+1$. This follows from Lemma~\ref{lem:algdegvec}, which implies that there are no subspaces $S$ such that $D_aD_b \pi = 0_m$, for all $a,b \in S$ and $\deg(\pi) > m-\dim(S)+1$.
   
   In the case $\deg(\pi)=m-1$, it follows that it is enough to check the property~\eqref{eq: P2} for the subspaces $S$ with $\dim(S) \leq 2$. Assume now that $S$ is a $1$-dimensional subspace of $\F_2^m$. Then the condition $D_aD_b \pi = 0_m$, for all $a,b \in S$ is trivially satisfied. Let $a \in \F_2^m$ be a nonzero vector such that $S= \langle a \rangle$. Since $\pi$ is a permutation, we have $D_a \pi(y) \neq 0_m$, for all $y \in \F_2^m$, and so $a$ is a linear structure of $\pi$ if and only if the subspace $L= \langle \{ D_a \pi(y) \neq 0_m \mid  y \in \F_2^m \} \rangle$ is $1$-dimensional. The subspace $L$ is $1$-dimensional if and only if $L^{\perp}$ is $(m-1)$-dimensional. Note that $L^{\perp}$ is such that 
   $$v \cdot D_a\pi(y) = 0;  \textnormal{ for all } y \in \F_2^m  \; \textnormal{, and  for all } v \in L^{\perp}.$$
   Hence, from the definition of \eqref{eq: P2}, we deduce that the vector $a$ is a linear structure of  $\pi$ if and only if $\pi$ does not satisfy the property~\eqref{eq: P2} with respect to the subspace $S=\langle a \rangle$. Consequently, the permutation $\pi$ has no linear structures if and only if $\pi$ satisfies the property~\eqref{eq: P2} for all $1$-dimensional subspaces, and the result follows.
\end{proof}
In~\cite[Remark 16]{PPKZ2023}, the authors provided examples of permutations that satisfy the \eqref{eq: P2} property but not~\eqref{eq: P1}. It was also noted that the property \eqref{eq: P1} implies \eqref{eq: P2}, though no formal proof of this fact was given. The following result  establishes this fact.
\begin{prop}\label{prop:implication}
	Let $\pi$ be a permutation of $\F_2^m$. If $\pi$ has the property \eqref{eq: P1}, then it also has the property \eqref{eq: P2}.
\end{prop}
\begin{proof}
The property~\eqref{eq: P1} implies \eqref{eq: P2} for the subspaces $S$ with $\dim (S) \geq 2$ trivially, since if a permutation $\pi$ satisfies \eqref{eq: P1} then there are no subspaces $S$ with $\dim (S) \geq 2$ such that $D_aD_b \pi =0_m$ for all $a,b \in S$. Assume now that $\dim (S) =1$. 
Using the same notation as in Proposition~\ref{prop:Vs(f)}, the permutation $\pi$ has  linear structures if and only if $\dim(V_S(f))=1$.
From the proof of Theorem \ref{th:refineddegreeP2}, it follows that the permutation $\pi$ has no linear structures if and only if $\pi$ satisfies
the property  \eqref{eq: P2} for all 1-dimensional subspaces.
 Furthermore, if a permutation $\pi$ satisfies the property \eqref{eq: P1}, then we deduce from Theorem~\ref{theo: unique P1} in \cite{PPKZ2023} that it has no linear structures, hence the property~\eqref{eq: P1} also implies \eqref{eq: P2} for the subspaces $S$ with $\dim (S) =1$, and consequently, we conclude that if a permutation $\pi$ satisfies \eqref{eq: P1}, then it also satisfies  the property~\eqref{eq: P2}.	
\end{proof}
\section{Constructing permutations satisfying the \eqref{eq: P2} property}\label{sec: 4 P2 property}
Finding more constructions of permutations with the \eqref{eq: P2} property was mentioned as an open problem in \cite{PPKZ2023}. In this section, we provide a solution to this problem by showing that adjusting the initial conditions on permutations $\sigma_1$ and $\sigma_2$ of $\F_2^m$ used in the following secondary construction of permutations with~\eqref{eq: P1} property one can construct permutations with \eqref{eq: P2} property.
\begin{prop}\label{prop:gensecderlarger}\cite{PPKZ2023}
	Let $\sigma_1$ and $\sigma_2$ be two permutations of $\F_2^m$ such that $D_V\sigma_1 \neq D_V\sigma_2$ for all 2-dimensional subspaces $V$ of $\F_2^m$. Define the function $\pi \colon \F_2^{m+1} \to \F_2^{m+1}$ by
	\begin{equation}\label{eq: P1-P2 secondary}
		\pi(y,y_{m+1})= \left( \sigma_1(y) + y_{m+1}(\sigma_1(y)+\sigma_2(y)) , y_{m+1} \right) \text{, for all } y \in \F_2^m, y_{m+1} \in \F_2.
	\end{equation}
	Then the function $\pi$ is a permutation of $\F_2^{m+1}$ such that  $D_aD_b\pi \neq 0_{m+1}$ for  all 2-dimensional subspaces $W=\langle a, b \rangle$ of $\F_2^{m+1}$, that is, $\pi$ satisfies the \eqref{eq: P1} property.
\end{prop}

A similar design method of preserving the property \eqref{eq: P2} was left as an open problem \cite[Open Problem 1]{PPKZ2023},  due to the complicated definition of this property. Employing the results in Section \ref{sec: 3 refining}, we provide a solution to this problem.
\begin{theo}\label{th: P2 secondary}
	Let $\sigma_1$ and $\sigma_2$ be two permutations of $\F_2^m$  
	and assume that $\sigma_1 + \sigma_2$ satisfies \eqref{eq: P2}. Let the function $\pi \colon \F_2^{m+1} \to \F_2^{m+1}$ be defined by Eq.~\eqref{eq: P1-P2 secondary}.
	Then, the function $\pi$ is a permutation of $\F_2^{m+1}$  that satisfies the \eqref{eq: P2} property.
\end{theo}
\begin{proof}
The fact that $\pi$ is a permutation of $\F_2^{m+1}$ follows from Proposition \ref{prop:gensecderlarger}. Assume that $\pi$ does not satisfy the \eqref{eq: P2} property.
Let $S \subset \F_2^{m+1}$ be a subspace of $\F_2^{m+1}$, $1 \leq \dim(S) \leq m$, such that $D_aD_b \pi =0_{m+1}$, for all $a,b \in S$, and let $V \subset \F_2^{m+1}$ be a subspace of $\F_2^{m+1}$ such that $\dim(S)+ \dim(V)=m+1$ and $v \cdot D_a \pi =0$, for all $v \in V$ and $a \in S$. 
From Lemma \ref{lem:algdegvec}, we know that $1 \leq \dim(S) \leq (m+1) -\deg(\pi)+1$.  It is clear that  $\deg(\pi) \geq 3 $ since $\sigma_1+\sigma_2$ satisfies  \eqref{eq: P2}.  Hence,  $1 \leq \dim(S) \leq m-1$ and  $\dim(V) \geq 2$ since $\dim(S)+\dim(V)=m+1$.
Let $S' = \lbrace \, s' \in \F_2^m  \mid (s',s_{m+1}) \in S \text{ for some }  s_{m+1} \in \F_2 \, \rbrace $ and $V' = \lbrace \, v' \in \F_2^m  \mid (v',v_{m+1}) \in V \text{ for some }  v_{m+1} \in \F_2 \, \rbrace $. 
Let $S_0 = \lbrace \, s' \in \F_2^m  \mid (s',0) \in S \, \rbrace $ and $V_0 = \lbrace \, v' \in \F_2^m  \mid (v',0) \in V \, \rbrace $.  From Equation \eqref{eq: P1-P2 secondary}, we calculate the derivatives

	\begin{eqnarray}\label{equTheo4.2 11}
	D_{(a',a_{m+1})} \pi(y,y_{m+1})= (D_{a'}\sigma_1(y) + y_{m+1}D_{a'}(\sigma_1 + \sigma_2)(y) & \nonumber \\ 
    + a_{m+1}(\sigma_1 + \sigma_2)(y+a'), \, a_{m+1} \,), &
	\end{eqnarray}

	\begin{eqnarray}\label{equTheo4.2 12}
	D_{(b', b_{m+1})}D_{(a', a_{m+1})}\pi(y,y_{m+1})= ( D_{b'}D_{a'}\sigma_1(y) + y_{m+1}D_{b'}D_{a'}(\sigma_1 + \sigma_2)(y) & \nonumber
    \\ + b_{m+1}D_{a'}(\sigma_1 + \sigma_2)(y+b') + a_{m+1}D_{b'}(\sigma_1 + \sigma_2)(y+a'), \, 0 ). &
	\end{eqnarray}

\noindent
From Equation \eqref{equTheo4.2 12}, it follows that $D_{b'}D_{a'}(\sigma_1 + \sigma_2)=0_m$, for all $a',b' \in S'$ (because of the variable $y_{m+1}$).
	Because $\dim(S) \geq \dim(S')\geq \dim(S)-1$, there are two cases to be considered depending on the dimension of $S'$.
	
	\begin{enumerate}[a)]
            \item If $\dim(S')=\dim(S)$, then, for all $v' \in V'$, from Equation \eqref{equTheo4.2 11} it follows that $v' \cdot D_{a'}(\sigma_1 + \sigma_2)=0$ (because of the variable $y_{m+1}$), for all $a' \in S'$. However, this implies that $\sigma_1 + \sigma_2$ does not satisfy the \eqref{eq: P2} property because $\dim(V') \geq \dim(V)-1= m- \dim(S')$.

            \item If $\dim(S')=\dim(S)-1$, then $\dim(S')=\dim(S_0)$, i.e., $S'=S_0$, hence $(0_m,1) \in S$. Since $\dim(V) \geq 2$, there exists a nonzero vector $v' \in V_0$. From Equation \eqref{equTheo4.2 11}, for $(a', a_{m+1})=(0_m,1)$, it follows that $v' \cdot (\sigma_1 + \sigma_2)=0$, and so $v' \cdot D_{w}(\sigma_1 + \sigma_2)=0$, for all $w \in \F_2^m$. However, this again implies that $\sigma_1 + \sigma_2$ does not satisfy the \eqref{eq: P2} property,  arriving at a contradiction again, and the result follows. \qedhere
        \end{enumerate}
\end{proof}
	
As already mentioned in the introduction, there exist 34 equivalence classes of quadratic permutations on $\F_2^5$ (out of the known 75, see~\cite{BBS2017}) that satisfy the property~\eqref{eq: P2}. Using these examples and Theorem~\ref{th: P2 secondary}, one can construct more permutations with~\eqref{eq: P2} property, as we illustrate in the following example.
	
	\begin{ex}
		Consider the following permutations on $\F_2^5$ that are given by their ANFs as follows:
		\begin{equation*}
			\sigma_1(y)=\begin{pmatrix}
				y_1\\
				y_2 + y_1 y_2 + y_1 y_3\\
				y_3 + y_1 y_3 + y_1 y_5\\
				y_1 y_2 + y_4 + y_1 y_4\\
				y_2 y_3 + y_1 y_4 + y_5 + y_1 y_5
			\end{pmatrix}^T \quad \mbox{and} \quad\sigma_2(y)= \begin{pmatrix}
				y_1 y_2 + y_1 y_5 + y_2 y_5\\
				y_1 + y_2 + y_2 y_5\\
				y_3 + y_1 y_4\\ 
				y_1 y_3 + y_4 + y_1 y_4\\
				y_1 + y_1 y_3 + y_3 y_4 + y_5 + y_2 y_5
			\end{pmatrix}^T.
		\end{equation*}
		The mapping $\sigma_1+\sigma_2$ has the~\eqref{eq: P2} property, though it is not a permutation, since, e.g, $|(\sigma_1+\sigma_2)(0)|=9$. By Theorem~\ref{th: P2 secondary}, the mapping $\pi(y,y_{m+1})= \left( \sigma_1(y) + y_{m+1}(\sigma_1(y)+\sigma_2(y)) , y_{m+1} \right)$, where $y \in \F_2^5$, $y_{m+1} \in \F_2$ is a permutation with~\eqref{eq: P2} property on $\F_2^6$.
	\end{ex}
	
	The following example indicates that the condition $\sigma_1 + \sigma_2$ satisfies the \eqref{eq: P2} property is only sufficient but not necessary for the mapping $\sigma$ to be a permutation with this property.
	
	\begin{ex}
		Let $\sigma_1$ be defined as in the previous example. Define the permutation $\sigma_3$ on $\F_2^5$ as follows:
		\begin{equation*}
			\sigma_3(y)=\begin{pmatrix}
				y_1\\
				y_2 + y_1 y_2 + y_1 y_4\\
				y_1 y_2 + y_3 + y_1 y_3\\
				y_2 y_3 + y_4 + y_1 y_4 + y_1 y_5\\
				y_1 y_3 + y_2 y_4 + y_5 + y_1 y_5
			\end{pmatrix}^T.
		\end{equation*} 
		Then, $\sigma_1+\sigma_3$ is given by
		\begin{equation*}
			(\sigma_1+\sigma_3)(y)=\begin{pmatrix}
				0\\
				y_1 y_3 + y_1 y_4\\
				y_1 y_2 + y_1 y_5\\
				y_1 y_2 + y_2 y_3 + y_1 y_5\\ 
				y_1 y_3 + y_2 y_3 + y_1 y_4 + y_2 y_4
			\end{pmatrix}^T.
		\end{equation*} 
		Note that the mapping $\pi(y,y_{m+1})= \left( \sigma_1(y) + y_{m+1}(\sigma_1(y)+\sigma_3(y)) , y_{m+1} \right)$, where $y \in \F_2^5$ and $y_{m+1} \in \F_2$, is a permutation with~\eqref{eq: P2} property on $\F_2^6$, though the mapping  $\sigma_1+\sigma_3$ does not have the~\eqref{eq: P2} property, as we indicate below. Let $S = \langle (0, 0, 1, 1, 0) \rangle$ and $V=\left\langle
		\scalebox{0.7}{$\begin{array}{ccccc}
				1 & 0 & 0 & 0 & 0 \\
				0 & 1 & 0 & 0 & 0 \\
				0 & 0 & 1 & 0 & 0 \\
				0 & 0 & 0 & 0 & 1 \\
			\end{array}$}
		\right\rangle$. We need to show that 
		$v \cdot D_a\pi(y) = 0;  \textnormal{ for all } a \in S, \; \textnormal{ all } y \in \F_2^m  \; \textnormal{, and  for all } v \in V$. The statement is obviously true for $a=0_5\in S$. For $a=(0, 0, 1, 1, 0)\in S$, we have that $D_a(\sigma_1+\sigma_3)(y)=(0, 0, 0, y_2, 0)$, and thus clearly $v \cdot D_a(\sigma_1+\sigma_3) = 0,  \textnormal{ for all }  v \in V$.
	\end{ex}

\section{Concatenation using ``swapping-like'' mappings -- a generalization}\label{sec: 5 swapping}
We first recall the following bent 4-concatenation approach considered in~\cite{PPKZ2023} (efficiently satisfying the dual bent condition), where the functions $f_i$ are defined below:
\begin{equation}\label{eq: bent4 change variables}
	\begin{split}
		f_1(x,y)&=f_2(x,y)=x \cdot \pi(y) + h_1(y),\\
		f_3(x,y)&=f_4(x,y)+1= y \cdot \sigma(x) + h_2(x),
	\end{split}
\end{equation}
for all $x,y \in \F_2^{m}$.

This approach was called ``swapping of variables'' in \cite{PPKZ2023} and is  in fact given by a linear transformation $L$, which maps the basis of the canonical $\mathcal{M}$-subspace so that $L\colon\langle I_m|O_m \rangle\to\langle O_m|I_m\rangle$ (where $I_m$ and $O_m$ stand for the all-zero and identity (binary) matrix of size $m \times m$, respectively). In a similar manner, one can introduce a series of linear mappings $L$ that can be  applied  to Maiorana-McFarland bent functions $f$ with the canonical $\mathcal{M}$-subspace $U$, in order to get bent functions $f'$ in $\mathcal{M}^\#$ with a unique $\mathcal{M}$-subspace $U'$, such that $U\cap U'$ is only of a ``small'' dimension. Then, similarly to Eq.~\eqref{eq: bent4 change variables}, one can concatenate functions
\begin{equation*}\label{eq: bent4 change variables generalized}
	\begin{split}
		f_1(x,y)&=f_2(x,y)=x \cdot \pi(y) + h_1(y)=f(x,y),\\
		f_3(x,y)&=f_4(x,y)+1=f'(x,y).
	\end{split}
\end{equation*}
The following result demonstrates that this is indeed possible. Notice also that any permutation $\pi$ that satisfies the property \eqref{eq: P2} can be used in our construction method. For convenience, we say that $V$ is an \textit{$\mathcal{M}_k$-subspace} of $f\in\mathcal{B}_n$ if it is an $\mathcal{M}$-subspace of $f$ of dimension $k$.  

\begin{theo}	\label{theo bent and q bent}
Let $n=2m$, and $x, y  \in \F_2^m$.
Let $\pi$ be a permutation of $\F_2^m$, $m \geq 3$ such that 
	$f_1(x,y)=x \cdot \pi(y) + h_1(y)$ is a bent  function with the unique canonical $\mathcal{M}_m$-subspace $\F_2^m\times\{0_m\}$. Let $P$ be a  permutation  over the set $\{1,2,\ldots,n\}$ such that there exists at least one element $i\in \{1,\ldots,m \}$ such that  $P(i)\notin \{1,\ldots,m \}$.
	Let $f_2(x, y)=f_1(x, y)$, 	$f_3(x, y)= f_1(x_{P(1)},\ldots,x_{P(n)})$,
	$f_4(x, y)=f_3(x, y)+1$, where $x=(x_{1}, \ldots, x_m )$ and $y=(x_{m+1}, \ldots, x_n )$.
	Set $f=f_1|| f_2||f_3||f_4$, which by Eq. \eqref{eq:ANF_4conc} gives
	\begin{equation*}\label{eq:formof4bent}
		f(x, y,z_1,z_2)=(1 +  z_1)f_1(x, y) +  z_1f_3(x, y) +  z_1z_2 ,\ (x, y)\in\F_2^n,z_1,z_2\in\F_2.
	\end{equation*}
	Then, $f \in \B_{n+2}$ is bent and outside  $\cM^{\#}$.
	
\end{theo}

\begin{proof}
	Since $f_1^* +  f_1^* +  f_3^* + (f_3 +  1)^*=1$, then $f$ is bent.
	
	For convenience, we denote $\mb{a}=(\mb{a}',a_{n+1},a_{n+2}),
	\mb{b}=(\mb{b}',b_{n+1},b_{n+2})\in {\Bbb F}_2^{n}\times {\Bbb F}_2\times {\Bbb F}_2$.
	Let $V$ be an arbitrary $(m+1)$-dimensional subspace of $\F_2^{n+2}$.
	From Lemma \ref{lem M-M second}, it is sufficient to show that for an arbitrary $(m+1)$-dimensional
	subspace $V$ of $\F_2^{n+2}$
	one can always find two vectors $\mb{a},\mb{b}\in { V}$ such that $  D_{(\mb{a}', a_{n+1}, a_{n+2})}D_{(\mb{b}', b_{n+1}, b_{n+2})}f(x, y,z_1,z_2)\neq 0 $, for some $(x, y,z_1,z_2)\in\vF{n+2}$.
	We have
		\begin{equation} \label{eq1: the main}
			\begin{array}{rl}
				&D_{(\mb{a}', a_{n+1}, a_{n+2})}D_{(\mb{b}', b_{n+1}, b_{n+2})}f(x, y,z_1,z_2)\\
				=& (1 +  z_1)D_{\mb{a}'}D_{\mb{b}'} f_1(x, y)  +  z_1D_{\mb{a}'}D_{\mb{b}'} f_3(x, y)  +  a_{n+1} D_{\mb{b}'} \left(f_1 +  f_3\right)(x, y) +  \mb{a}')\\
			+	&  b_{n+1}D_{\mb{a}'}\left(f_1 +  f_3\right)(x, y) +  \mb{b}')  +  a_{n+1}b_{n+2} +  a_{n+2}b_{n+1}.
			\end{array}
		\end{equation}
	There are two cases to be considered. 
	\begin{enumerate}
		\item   We first assume  that  $\dim\left(V\cap (\F_2^n\times \{(0,0)\})\right)\geq m$, which  will imply the existence of  two vectors $\mb{a}=(\mb{a}',a_{n+1},a_{n+2}),
		\mb{b}=(\mb{b}',b_{n+1},b_{n+2})\in V$ such that $ \mb{a}'\neq \mb{b}'$,   $a_{n+1}=a_{n+2}=b_{n+1}=b_{n+2}=0$,  for which  $D_{\mb{a}'}D_{\mb{b}'} f_3\not\equiv 0$ or $D_{\mb{a}'}D_{\mb{b}'} f_1\not\equiv 0$,
			 as shown below.
			
		Namely, from the definition of $P$ and $f_3$, we know that 
			$f_3$  also has  a unique $\mathcal{M}_m$-subspace $U'$ and $U\neq U'$, assuming that $f_1$ has the unique canonical $\mathcal{M}_m$-subspace $U$. 
			
			Thus, we must have 
			$$ \left(V\cap (\F_2^n\times \{(0,0)\})\right)\setminus U'\neq \emptyset,$$
			or
			$$ \left(V\cap (\F_2^n\times \{(0,0)\})\right)\setminus U\neq \emptyset,$$
			since $\dim\left(V\cap (\F_2^n\times \{(0,0)\})\right)\geq m$ and $U\neq U'$.
			
			If 
			$$ \left(V\cap (\F_2^n\times \{(0,0)\})\right)\setminus U'\neq \emptyset,$$
			then we can 
			find two vectors $\mb{a}=(\mb{a}',a_{n+1},a_{n+2}),
			\mb{b}=(\mb{b}',b_{n+1},b_{n+2})\in V$ such that $ \mb{a}'\neq \mb{b}'$,   $a_{n+1}=a_{n+2}=b_{n+1}=b_{n+2}=0$,  and moreover  $D_{\mb{a}'}D_{\mb{b}'} f_3\not\equiv 0$ since 	$f_3$  has   a unique $\mathcal{M}_m$-subspace $U'$. 
			
			From Eq. \eqref{eq1: the main},  for $z_1=1$, we obtain
			$$D_{(\mb{a}', a_{n+1}, a_{n+2})}D_{(\mb{b}', b_{n+1}, b_{n+2})}f(x, y,1,z_2)= D_{\mb{a}'}D_{\mb{b}'} f_3(x, y)\not\equiv 0.$$

			Now, assume that 	$$ \left(V\cap (\F_2^n\times \{(0,0)\})\right)\setminus U\neq \emptyset.$$
			Similarly,  there will exist two vectors $\mb{a}=(\mb{a}'',a_{n+1},a_{n+2}),
			\mb{b}=(\mb{b}'',b_{n+1},b_{n+2})\in V$ such that $ \mb{a}''\neq \mb{b}''$,   $a_{n+1}=a_{n+2}=b_{n+1}=b_{n+2}=0$,  for which  $D_{\mb{a}''}D_{\mb{b}''} f_1\not\equiv 0$. Setting  $z_1=0$ in Eq.  \eqref{eq1: the main}, we obtain
			$$ D_{(\mb{a}', a_{n+1}, a_{n+2})}D_{(\mb{b}', b_{n+1}, b_{n+2})}f(x, y,0,z_2)=D_{\mb{a}'}D_{\mb{b}'} f_1(x, y)\not\equiv 0,$$
			and again  we conclude that $D_{\mb{a}}D_{\mb{b}}f(x, y,z_1,z_2)\neq 0$. 
		\item When $\dim\left(V\cap (\F_2^n\times \{(0,0)\})\right)< {m}$, 
		we have $V\cap (\{\mb{0}_n\}\times \F_2^2)=\{0_n\} \times \F_2^2$  since $$\dim\left(V\cap (\F_2^n\times \F_2^2) \right)=m+1.$$
		Furthermore, we can find two vectors $\mb{a}=(\mb{a}',a_{n+1},a_{n+2}),
		\mb{b}=(\mb{b}',b_{n+1},b_{n+2})\in V$ such that $ \mb{a}'=\mb{0}_n, \mb{b}'=\mb{0}_n$, $a_{n+1}=1, b_{n+1}=0$, and $a_{n+2}=0, b_{n+2}=1$. From Eq. \eqref{eq1: the main}, we have
		\begin{equation*} \label{eq1: the main 1}
			\begin{array}{c}
				D_{(\mb{0}_n, 1, 0)}D_{(\mb{0}_n, 0, 1)}f(x, y,z_1,z_2)
				= 1\neq 0.
			\end{array}
		\end{equation*}
	\end{enumerate}
	Thus, there is no $(m+1)$-dimensional linear subspace of $\vF{n+2}$ on which the second-order derivatives of $f$ vanish, i.e., $f$ is outside $\cM^{\#}$.
\end{proof}

The main condition in Theorem \ref{theo bent and q bent} is that the functions $f_1$ and $f_3$ do not share their unique $\cM$-subspaces of maximal dimensions (thus ensuring that $U \neq U'$). More generally, instead of a suitable permutation of indices used in Theorem \ref{theo bent and q bent}, the same goal can be achieved by properly selecting linear permutations $L: \F_2^m \times \F_2^m \to  \F_2^m \times \F_2^m$, as stated below.
\begin{cor}\label{cor:swapping}
	Let $n=2m$  and $(x, y)\in \F_2^m\times \F_2^m$.  Let $\pi$ be a permutation of $\F_2^m$, $m \geq 3$ such that 
	$f_1(x, y)=x \cdot \pi(y) + h_1(y)$ is a bent  function with the unique canonical $\mathcal{M}_m$-subspace $U=\F_2^m\times\{0_m\}$. Let $L$ be a linear permutation  over $\F_2^n$ such that $f_1(L(x, y)) $ has  a unique $\mathcal{M}_m$-subspace $U'$ and $U'\neq U$.
	Define $f_i \in \B_n$, for $i=2,3,4$, as: 
	$$f_2(x, y)=f_1(x, y), \;\; 	f_3(x, y)= f_1(L(x, y)), \; \;f_4(x, y)=f_3(x, y)+1.$$ 
	Then,  $f=f_1|| f_2||f_3||f_4 \in \B_{n+2}$, 
	is bent and outside  $\cM^{\#}$.
\end{cor}

	\begin{ex}
		Consider the following permutation on $\F_2^3$ given by its algebraic normal form:
		\begin{equation*}
			\pi(y)=\begin{pmatrix}
				y_1 + y_2 + y_3 + y_2 y_3 \\
				y_2 + y_1 y_2 + y_1 y_3 \\
				y_1 y_2 + y_3
			\end{pmatrix}^T,
		\end{equation*}
		which describes the inverse function $\pi(y)=y^{-1}$ on $\F_{2^3}$. Let the linear permutation  $L\colon\F_2^6\to\F_2^6$ be given by $L(x_1,x_2,x_3,y_1,y_2,y_3)=(x_1, x_2, y_1, x_3, y_2, y_3)$. Define the bent functions $f_1,f_2,f_3,f_4\in\mathcal{B}_6$ as $f_1(x,y)=f_2(x,y)=x\cdot\pi(y)$, $f_3(x,y)=f_1(L(x,y))$ and $f_4(x,y)=f_3(x,y)+1$, for $x,y\in\F_2^3$. Note that every $f_i$ has a unique $\mathcal{M}$-subspace of dimension 3, since $\pi$ is APN (thus having the ~\eqref{eq: P1} property). The ANF of $f=f_1|| f_2||f_3||f_4 \in \B_{8}$ is given by
		\begin{polynomial}
			f(z)=z_1 z_4 + z_1 z_5 + z_2 z_5 + z_2 z_4 z_5 + z_3 z_4 z_5 + z_1 z_6 + z_3 z_6 + 
			z_2 z_4 z_6 + z_1 z_5 z_6 + z_1 z_3 z_7 + z_1 z_4 z_7 + z_2 z_3 z_5 z_7 + 
			z_2 z_4 z_5 z_7 + z_3 z_6 z_7 + z_2 z_3 z_6 z_7 + z_4 z_6 z_7 + z_2 z_4 z_6 z_7 + z_7 z_8,
		\end{polynomial}
		\noindent where $z=(z_1,\ldots,z_{8})=(x,y,z_7,z_{8})$. By Theorem \ref{theo bent and q bent} or  Corollary~\ref{cor:swapping}, $f=f_1|| f_2||f_3||f_4 \in \B_{8}$ is bent and outside  $\cM^{\#}$, which was also confirmed using Magma.
	\end{ex}

\section{$\ell$-Optimal bent functions and their construction methods} \label{sec: 6 loptimal}
In a series of recent articles~\cite{Cclass_DCC,OutsideMM,BFAExtended,BapicPasalic,Kudin2021,Kudin2022,PKP_WCC_2024}, the authors provided new design methods of bent functions $f\in\mathcal{B}_n$ outside $\cM^\#$. The latter is equivalent to the fact that the maximal dimension of an $\mathcal{M}$-subspace of $f$ (recall that this number is called the linearity index of $f$ and denoted by $\operatorname{ind}(f)$) is strictly less than $n/2$. In this section, we explain how to specify bent functions with a prescribed linearity index less than $n/2$ using the permutations with~\eqref{eq: P1} property. In this way, we  provide bent functions that can be used recursively for constructing the new ones using the following result.
\begin{cor}\label{cor:mixingtMMandoutsideMM} \cite[Corollary 33]{PPKZ2023} Let $f_1$ be an arbitrary  bent function on $\F_2^n$ in $\cM^\#$, and let $f_3$ be a bent function on $\F_2^n$ that only admits $\mathcal{M}$-subspaces of dimension strictly less than $n/2-1$. 
		Set $f_2=f_1$ and $f_4=1+f_3$. Then, $f=f_1||f_2||f_3||f_4 \in \B_{n+2}$ is a  bent function outside $\cM^\#$.
\end{cor}

As mentioned in Subsection~\ref{sub: 2.1 MM}, for a Boolean bent function $f\in\mathcal{B}_n$, its \textit{linearity index} satisfies $1\le \operatorname{ind}(f)\le n/2$. Moreover, $\operatorname{ind}(f)=n/2$ if and only if $f\in\mathcal{M}^\#$. In view of these observations, it is natural to consider bent functions with the minimal linearity index as the ``counterparts'' of bent functions from the $\mathcal{M}^\#$ class.

\begin{defi}
	Let $f\in\mathcal{B}_n$ be bent. If $\operatorname{ind}(f)=1$, we say that $f$ is \textit{$\ell$-optimal}, i.e., $f$ has optimal linearity index. 
\end{defi}

In the remaining part of this section, we provide constructions of $\ell$-optimal bent functions using permutations with~\eqref{eq: P1} property and consider their 4-decomposition.

\subsection{$\ell$-Optimality and $\cM$-subspaces of $\cD_0$ functions}\label{sub: 6.1. l-optimality and D0}
In this section, we analyze the possibility of specifying bent functions with the lowest possible linearity index, i.e. $\ell$-optimal bent functions. The following result will be useful for our purpose, where  we use $\mathbbm{1}_W(x)$ to denote the indicator function of the subspace $W$ (thus $\mathbbm{1}_W(x)=1$, if $x \in W$ and zero otherwise).

\begin{lemma} \label{lem: indder} Let $V$ and $W$ be two vector subspaces of $\F_2^m$. If $V \cap W = \{ 0_m \}$, then $$D_V \mathbbm{1}_W(x)= \mathbbm{1}_{V \oplus W}(x), \text{ for all } x \in \F_2^m.$$ 
If $V \cap W \neq \{ 0_m \}$, then $D_V \mathbbm{1}_W(x)=0$, for all $x \in \F_2^m$.
\end{lemma}
\begin{proof}
Assume first that $V \cap W \neq \{ 0_m \}$ and let $v_1$ be a nonzero vector in $V \cap W$. Let $\{v_1, \ldots ,v_k \}$ be a basis for $V$ containing the vector $v_1$. Since $v_1$ is also in $W$, we have that $x$ is in $W$ if and only if $x+v_1$ is in $W$, for all $x \in \F_2^m$, and consequently $\mathbbm{1}_W(x)=\mathbbm{1}_W(x+v_1)$, for all $x \in \F_2^m$. Hence, 
$$D_V \mathbbm{1}_W(x)= D_{v_n} \cdots D_{v_2}D_{v_1} \mathbbm{1}_W(x)= D_{v_n} \cdots D_{v_2}(\mathbbm{1}_W(x)+\mathbbm{1}_W(x+v_1)) = D_{v_n} \cdots D_{v_2}(0)=0,$$
for all $x \in \F_2^m$.

Assume now that $V \cap W = \{ 0_m \}$. Then, every vector $z$  in $V \oplus W$ can be represented in a unique way as $z=v+w$, for some $v \in V$ and $w \in W$. Notice that for any subspace $U$ of $\F_2^m$ we have that $\mathbbm{1}_U(x)= \sum_{u \in U} \delta_0(x+u)$, for all $x \in \F_2^m$, and so we compute
\begin{equation*} 
\begin{split}
D_V \mathbbm{1}_W(x)  = D_V \left( \sum_{w \in W} \delta_0(x+w) \right) & = \sum_{w \in W} D_V \left( \delta_0(x+w) \right) = \sum_{w \in W} \sum_{v \in V} \delta_0(x+v+w) \\
& = \sum_{z \in V \oplus W} \delta_0(x+z) = \mathbbm{1}_{V \oplus W }(x),
\end{split}
\end{equation*}
for all $x \in \F_2^m$, and this concludes the proof.
\end{proof}

As a consequence of Lemma~\ref{lem: indder}, we have that for any two linearly independent vectors $a,b$ in $\F_2^m$ the second-order derivative of the indicator function $\delta_0=\mathbbm{1}_{0_m}$ (in the direction governed by $a$ and $b$) is $D_aD_b\delta_0=\mathbbm{1}_{\langle a,b \rangle}$, where $\langle a,b \rangle$ is the subspace of $\F_2^m$ generated by $a$ and $b$, i.e., $\langle a,b \rangle = \{ 0_m, a, b, a+b \}$.  To consider $\mathcal{M}$-subspaces of bent functions from $\cD_0$ class introduced by Carlet~\cite{CC93}, we need the following definitions introduced in \cite{Polujan2020}.
\begin{defi}\cite{Polujan2020}\label{defPolujanandPott2020}
		A vector subspace $U  \subseteq \F_2^n$ is called a relaxed $\mathcal{M}$-subspace of a Boolean function $f\in \mathcal{B}_n$, if for all $a,b \in U$ the second-order derivatives $D_{a}D_bf$ are either constant zero or constant one functions, i.e., $D_{a}D_bf=0$ or $D_{a}D_bf=1$. We denote by $\mathcal{RMS}_r(f)$ the collection of all $r$-dimensional relaxed $\mathcal{M}$-subspaces of a Boolean function $f$ and by $\mathcal{RMS}(f)$ the collection $\mathcal{RMS}(f):=\bigcup\limits_{r=1}^{n}\mathcal{RMS}_r(f)$.
		For a Boolean function  $f\in \mathcal{B}_n$ its relaxed linearity index $r$-ind($f$) is defined by $r$-ind($f$)$:=\max\limits_{U\in \mathcal{RMS}(f) }\dim(U).$
\end{defi} 
With this definition, and the fact that for a Boolean function $f\in\mathcal{B}_n$ it holds that $\operatorname{ind}(f)\leq\operatorname{r-ind}(f)$, see~\cite{Polujan2020}, we are ready to prove the main result of this section.
\begin{theo}\label{theo: MofDzero} Let $\pi$ be a permutation of $\F_2^m$, $m \geq 4$ which has the property \eqref{eq: P1}. Define $f\colon\F_2^{m}\times\F_2^{m} \to \F_2$ by $f(x,y)=x \cdot \pi(y) + \delta_{0}(x)$, for all $x,y \in \F_2^{m}$. 
	Then, $\operatorname{r-ind}(f)\leq 2$. Furthermore, $\operatorname{r-ind}(f)=1$, implying that $\operatorname{ind}(f)=1$, if and only if $\pi$ has no components with linear structures.
\end{theo}
\begin{proof}
	Let $V$ be an $\cM$-subspace of $f$. Let $L_V \to \F_2^{m}$ be the projection $L(x,y)=y$, for all $(x,y) \in V$.
		In general, denoting $a=(a_1,a_2)$ and $b=(b_1,b_2)$ in $\F_2^{2m}$, we have
	\begin{equation} \label{eq: secder2}
	D_{(a_1, a_2)}D_{(b_1, b_2)}f(x,y)
	=x\cdot\left( D_{a_2}D_{b_2}\pi(y)\right) +a_1\cdot D_{b_2}\pi(y+ a_2)
	+ b_1 \cdot D_{a_2}\pi(y+ b_2) 
	+ D_{a_1}D_{b_1}\delta_0(x). 
	\end{equation}
	Since for linearly independent $a_1,b_1$, by Lemma \ref{lem: indder}, we have  $D_{a_1}D_{b_1}\delta_0 = \mathbbm{1}_{\langle a_1,b_1 \rangle}$, then the algebraic degree of $D_{a_1}D_{b_1}\delta_0(x)$ is $m-2 \geq 2$. On the other hand, if $a_1,b_1$ are linearly dependent then $D_{a_1}D_{b_1}\delta_0=0$.  Thus,  to force  $D_{(a_1, a_2)}D_{(b_1, b_2)}f$ to be zero, from Eq. \eqref{eq: secder2}, we deduce that it has to be the case that $x\cdot\left( D_{a_2}D_{b_2}\pi(y) \right) = 0$, that is $D_{a_2}D_{b_2} \pi =0_m$. Because the permutation $\pi$ has the \eqref{eq: P1} property, we conclude that $\dim (Im(L)) \leq 1$.
	
	For $(a_1,0_m)$ and $(b_1,0_m)$ in $Ker(L)$, we have $D_{(a_1, 0_m)}D_{(b_1, 0_m)}f(x,y) = D_{a_1}D_{b_1}\delta_0(x)$, and hence the subspace $\langle a_1,b_1 \rangle$ is at most $1$-dimensional, (otherwise the algebraic degree of $D_{a_1}D_{b_1}\delta_0(x)$ is $m-2$, so $D_{(a_1, 0_m)}D_{(b_1, 0_m)}f \neq 0$, a contradiction). Consequently, we deduce that $\dim(Ker(L)) \leq 1$, and from the rank-nullity theorem, it follows that $\dim(V)= \dim(Ker(L)) + \dim(Im(L)) \leq 1+1=2$.
	
	 Assume now that $a=(a_1,a_2)$ and $b=(b_1,b_2)$ are two linearly independent vectors from $V$. 
	From the first part of the proof, we know that $a_2$ and $b_2$ are linearly dependent, and because $D_{(a_1, a_2)}D_{(b_1, b_2)}f= D_{(a_1, a_2)+(b_1, b_2)}D_{(b_1, b_2)}f$, we can without loss of generality assume that $a_2=0_m$. From Eq. \eqref{eq: secder2}, it follows that 
	$
	D_{(a_1, 0_m)}D_{(b_1, b_2)}f(x,y)
	= a_1\cdot D_{b_2}\pi(y) 
	+ D_{a_1}D_{b_1}\delta_0(x),
	$
	hence we can also deduce that $a_1$ and $b_1$ are linearly dependent (similarly as above, otherwise the algebraic degree of $D_{a_1}D_{b_1}\delta_0(x)$ is $m-2 \geq 2$, so $D_{(a_1, 0_m)}D_{(b_1, b_2)}f$ is not a constant function, a contradiction). Since $a_1 \neq 0_m$, and because $D_{(a_1, 0_m)}D_{(b_1, b_2)}f= D_{(a_1, 0_m)}D_{(a_1, 0_m)+(b_1, b_2)}f$, we can without loss of generality assume that $b_1=0_m$. This means that for every $2$-dimensional relaxed $\cM$-subspace $V$ of $f$ we can find a basis $\{a,b \}$ for $V$ of the form $a=(a_1,0_m)$ and $b=(0_m, b_2)$. From Eq. \eqref{eq: secder2}, it follows that 
	$D_{(a_1, 0_m)}D_{(0_m, b_2)}f(x,y) =a_1\cdot D_{b_2}\pi(y),$
	and consequently, $a_1\cdot D_{b_2}\pi(y)$ is a constant function, hence the component $a_1 \cdot \pi$ of $\pi$ has a nonzero linear structure $b_2$. On the other hand, if $\pi$ has a component $a_1 \cdot \pi$, $a_1\neq 0_m$,  with a nonzero linear structure $b_2$, then it follows from Eq. \eqref{eq: secder2} that the subspace $\langle (a_1,0_m), (0_m,b_2) \rangle$ is a $2$-dimensional relaxed $\cM$-subspace of $f$.
\end{proof}

Theorem~\ref{theo: MofDzero} provides a method for obtaining functions  satisfying $\operatorname{r-ind}(f)< n/2$, that is, it gives a solution to~\cite[Open Problem 1]{ZPBW23}. We notice that the bent functions within $\cD_0$ can satisfy $\ell$-optimality, provided that a permutation $\pi$ satisfies the \eqref{eq: P1} property and additionally the components of $\pi$ do not admit linear structures. Therefore, we have the following corollary related to Theorem \ref{theo: MofDzero}.

\begin{cor}\label{cor:loptimal} Let $\pi$ be a permutation of $\F_2^m$, $m \geq 4$ which has the property \eqref{eq: P1} and additionally satisfies the condition that $ a_1 \cdot  D_{b_2}\pi(y) \neq 0$, for any nonzero $a_1,b_2 \in \F_2^m$.  Then, the bent function $f(x,y)=x \cdot \pi(y) + \delta_{0}(x)$, where $x,y \in \F_2^{m}$, is outside 
	$\cM^\#$ and also $\ell$-optimal since its $\operatorname{r-ind}(f)=1$. In particular, the same is true if  $\deg(\pi)>2$ and $\pi$ is a monomial permutation satisfying \eqref{eq: P1}.
	
\end{cor}
\begin{proof}
	The first part follows immediately from Theorem~\ref{theo: MofDzero}. It was shown in \cite{Pascale_Gohar2010} that the components of monomial permutations of degree $>2$ do not admit linear structures, and the second part of the statement follows. 
\end{proof}
\begin{rem}\label{rem:P2+nolinstruct}
	\emph{1.} The permutations $\pi$ that preserve the property  \eqref{eq: P1} on larger variable spaces, as in Proposition \ref{prop:gensecderlarger}, cannot satisfy the condition  that $a_1 \cdot  D_{b_2} \pi(y)\neq 0$, due to the last coordinate function of $\pi$ which is $y_{m+1}$. Thus, it is of interest to provide similar extensions preserving the property \eqref{eq: P1} without adding linear terms. \ \\
	\noindent\emph{2.} We notice that the property \eqref{eq: P2} along with the condition that $a_1 \cdot  D_{b_2}\pi(y)\neq 0$ is not sufficient for specifying $\ell$-optimal bent functions in Theorem \ref{theo: MofDzero}. Indeed, assuming that a permutation $\pi$ on $\F_2^m$ does not have \eqref{eq: P1}, then one can find two vectors $a, b \in \F_2^m$ such 
	that $D_aD_b\pi=0$.  The subspace generated by $(0_m,a)$ and $(0_m,b)$ is a 2-dimensional $\mathcal{M}$-subspace of $f(x,y)=x \cdot \pi(y) + \delta_0(x)$, and thus $f$ is not $\ell$-optimal.
\end{rem}

\begin{ex}\label{ex: 5-valued l-optimal}
	\textit{1.} In his thesis, Dillon showed that the  partial spread bent function $f(x)=Tr(x^{15})$ satisfies: $\deg(D_aD_bf)=2$ for any 2-dimensional vector space $\langle a, b \rangle$, see~\cite[pp. 102-103]{Dillon}. In turn, it means that $\operatorname{ind}(f)=1$, since $D_aD_bf=0$ iff $\dim\langle a, b \rangle\le1$, thus, this function is $\ell$-optimal. We also note that all possible 4-decompositions of these functions are 5-valued.
	
	\noindent{2.} Let $f(x,y)=Tr(xy^{-1})+\delta_0(x)$ be a bent function on $\F_{2^5}\times\F_{2^5}$. Since $y\mapsto y^{-1}$ is an APN permutation on  $\F_{2^5}$, we have that the inverse permutation has $(P_1)$ property. Moreover, the components of a monomial permutation whose degree is larger than 2 do not admit linear structures, see \cite{Pascale_Gohar2010}.    By Theorem~\ref{theo: MofDzero},  $\operatorname{r-ind}(f)=1$ and thus $f$ is $\ell$-optimal . With a computer algebra system, one can show that the multiset of degrees of the second-order derivatives of $f$ is given by:
	$$\{*\;\deg(D_aD_bf)\colon \dim\langle a,b \rangle=2\;*\}=\{*\; 3^{174\,220},\; 2^{31}  \;*\},$$
	where $3^{174\,220}$ means that there are ${174\,220}$ 2-dimensional subspaces $\langle a,b \rangle$ such that $\deg(D_aD_bf)=3$,  and similarly $2^{31}$ indicates that there are ${31}$ subspaces $\langle a,b \rangle$ such that $\deg(D_aD_bf)=2$. Notice that $$ (2^n - 1)(2^n -2)/6= 174251=174220+31,$$ when $n=10$.
	This confirms that $f$ is $\ell$-optimal.
	We also note that all possible 4-decompositions of these functions are 5-valued.
\end{ex}

The above remarks and examples lead naturally to the following research problems.

\begin{op}\label{op:P2+nolinstruc}
	Find generic constructions of permutations $\pi$ on $\F_2^m$ satisfying the \eqref{eq: P1} property whose components do not admit linear structures.
\end{op}

\begin{op}\label{op:distributionderiv}
	Provide theoretical estimates on the value distribution of the multiset of derivatives (for a varying degree of $D_aD_bf$) 
	$$\{*\;\deg(D_aD_bf)\colon \dim\langle a,b \rangle=2\;*\},$$
	which is an interesting but challenging research task. 
\end{op}

\subsection{4-Decomposition of bent functions in $\cD_0$}\label{sub: 6.2. 4-deom and D0}
Any bent function $f\in\mathcal{B}_{n+2}$ has $(2^{n+2}-1)(2^{n+2}-2)/6$ different 4-decompositions $f = (f_1, f_2, f_3, f_4)_V$ into bent, semi-bent, or 5-valued spectra functions $f_i\in\mathcal{B}_{n}$, where $V=\langle a,b \rangle^\perp$; these different 4-decompositions correspond to  2-dimensional subspaces $\langle a,b\rangle$ of $\F_2^{n+2}$. To the best of our knowledge, the only known $\ell$-optimal bent functions (apart from those constructed in the previous subsection) are monomials $Tr(\lambda x^k)$ on $\F_{2^n}$, where $n,\lambda$ and $k$ are suitably chosen, see~\cite[Lemma 3]{Decom}. More precisely, Charpin and Canteaut in~\cite[Theorem 10]{Decom} proved that  monomials $Tr(\lambda x^k)$ on $\F_{2^n}$, where $n,\lambda$ and $k$ are chosen as in~\cite[Lemma 3]{Decom} have neither  bent nor semi-bent 4-decompositions,  hence all  $(2^n - 1)(2^n -2)/6$ decompositions are 5-valued. More $\ell$-optimal bent functions with such properties were given in Example~\ref{ex: 5-valued l-optimal}. 

It is well-known \cite{Decom} that  a bent function $f\in\mathcal{B}_n$ has only 5-valued spectra decompositions if and only if  $D_aD_b f^*\neq const.$, for all 2-dimensional subspaces $\langle a,b\rangle$, and the latter is equivalent to the fact  that $\operatorname{r-ind}(f^*)=1$, which implies that $\operatorname{ind}(f^*)=1$, that is,  $f^*$ is $\ell$-optimal. In view of this conclusion and the discussion above, it is natural to conjecture  that an $\ell$-optimal bent function only has 5-valued 4-decompositions. In Theorem~\ref{theo:5valueddecomp}, we show that certain infinite families of bent functions indeed have this property. However, in Remark~\ref{rem: not 5-valued only} we indicate that this statement is not true in general.

\begin{theo}\label{theo:5valueddecomp}
	Let $m \geq 4$ and $\pi$ be a monomial  APN permutation on $\F_2^m$, with $\deg(\pi) >2$, which induces an $\ell$-optimal bent function $f(x,y)=x\cdot \pi(y)+\delta_0(x)$ on $\F_2^m\times\F_2^m$. Then, provided that $\deg(\pi^{-1}) >2$, its dual is also an $\ell$-optimal bent function and allows only 5-valued spectra decompositions. In particular, the inverse function $\pi(y)=y^{-1}$ on $\F_2^m$ (where by convention $\pi(0)=0$) is an example of such a permutation. 
\end{theo}
\begin{proof}
The dual of $f(x,y)=x\cdot \pi(y)+\delta_0(x)$,  is given by $f^*(x,y)=y\cdot\pi^{-1}(x)+\delta_0(y)$, see \cite[p. 211]{Carlet2021}. Since the inverse 
of an APN function is also APN \cite{CCZ1998}, then $\pi^{-1}(x)$ is also a monomial APN permutation and $f^*\in \cD_0$. Since we have assumed that $\deg(\pi^{-1}) >2$, then the components of $\pi^{-1}$ do not admit linear structures, see \cite{Pascale_Gohar2010}. Hence,  $\operatorname{r-ind}(f^*)=1$ by Corollary \ref{cor:loptimal}, which is equivalent to the property of admitting 5-valued spectra decompositions only. 

\end{proof}
\begin{rem}\label{rem: not 5-valued only}
	However, if $\deg(\pi^{-1}) =2$ in Theorem \ref{theo:5valueddecomp} then the components of $\pi^{-1}$ admit linear structures and $f^*(x,y)=y\cdot\pi^{-1}(x)+\delta_0(y)$ is not $\ell$-optimal, hence $f$ has 4-decompositions that are different from 5-valued ones. For instance, taking 
	$\pi(y)=y^{21}$ on $\F_2^5$ to specify $f(x,y)=x\cdot \pi(y)+\delta_0(x)$, it can be easily verified that $\pi^{-1}(x)=x^{3}$ and the function $f^*(x,y)=y\cdot\pi^{-1}(x)+\delta_0(y)$ is not $\ell$-optimal, which implies that $f$ does not admit 5-valued spectra decompositions only.
\end{rem} 

It is easy to see that any APN permutation (including the inverse $y\in\F_{2^m}\mapsto y^{-1}$, for $m$ odd) has the~\eqref{eq: P1} property. In the following statement, we indicate that for $m$ even, the inverse permutation has the~\eqref{eq: P1}  property as well.
\begin{prop}
	Let $m$ be even and $\pi(y)=y^{-1}$ be the inverse permutation on $\F_{2^m}$. Then, $\pi$ has the~\eqref{eq: P1} property.
\end{prop}
\begin{proof}
	Recall that by ~\cite[Theorem III.3]{Li20}, we have that $D_{a,b}\pi(y)=\pi(y)+\pi(y+a)+\pi(y+b)+\pi(y+a+b)=0$ for $\{y,y+a,y+b,y+a+b\}\in\mathcal{VB}_{m,-1}$, which is defined by $\mathcal{VB}_{m,-1}:=\left\{ \left\{0,\alpha ^i,\zeta  \alpha ^i,\zeta ^2 \alpha ^i\right\} \colon 0\le i \le \frac{2^m-4}{3}  \right\}$, where $\alpha$ is a primitive element of $\F_{2^m}$ and $\zeta=\alpha^\frac{2^m-1}{3}$. 
	 Then, clearly $D_{a,b}\pi(y)=0$ has four solutions $y\in \{0,a,b,a+b\}$ if and only if $\langle a,b \rangle \in \mathcal{VB}_{m,-1}$, and zero solutions otherwise. Consequently, $D_aD_b\pi$ is not the constant zero function for all linearly independent  $a,b\in\F_{2^m}$, and hence the inverse permutation $\pi$ on $\F_{2^m}$ has the~\eqref{eq: P1} property, when $m$ is even as well.
\end{proof}

To conclude this section, we propose the following open problem about $\ell$-optimal bent functions.

\begin{op} \label{op:loptimaldualaswell}
Specify other generic classes of  $\ell$-optimal bent functions $f\in\mathcal{B}_n$. Particularly,  find $\ell$-optimal bent functions $f\in\mathcal{B}_n$ such that   $f^*\in\mathcal{B}_n$ is also $\ell$-optimal.
\end{op}

\section{Conclusion and open problems}\label{sec: 7 concl}
In this article, we have  further refined the two important properties of permutations, the so-called \eqref{eq: P1} and  \eqref{eq: P2} properties, which are used in the construction of bent functions in $\cM$ that admit a unique $\mathcal{M}$-subspace of maximal dimension. We generalize the constructions methods of such permutations compared to \cite{PPKZ2023}, which is useful in the context of extending the design methods of bent functions that are provably outside $\cM^\#$.  Additionally, we generalize the so-called ``swapping variables'' method introduced in \cite{PPKZ2023} which then allows us to easily specify much larger families of bent functions outside $\cM^\#$ compared to \cite{PPKZ2023}.  We also specify $\ell$-optimal bent functions within the $\cD_0$ class whose linearity index is the lowest possible. Such bent functions can be employed in certain secondary constructions of bent functions \cite{ZPBW23} for providing further classes of bent functions that are provably outside $\cM^\#$. Moreover, we demonstrate that a certain subclass of $\cD_0$ has an additional property of having only 5-valued spectra decompositions, similarly to the only result in this direction given in \cite{Decom}.

There are several open problems that are left unanswered, where  in particular generic constructions of permutations that satisfy  the \eqref{eq: P1} property 
and whose components do not admit linear structures is of great importance for specifying other $\ell$-optimal bent functions.

\section*{Acknowledgments} 
Enes Pasalic is supported in part by the Slovenian Research Agency (research program P1-0404
and research projects J1-4084 and J1-60012). Sadmir Kudin is supported in part by the Slovenian Research Agency (research program P1-0404 and research project J1-60012).  Fengrong Zhang is supported by in part by   the Natural Science Foundation of China under Grant (62372346),  in part by the Higher Education Discipline Innovation Introduction Plan (B16037).




\end{document}